\documentclass{amsart}

\usepackage{amsmath}
\usepackage{amssymb}
\usepackage{amscd}
\usepackage{amsthm}
\usepackage{tikz}
\usetikzlibrary{positioning}
\usepackage[colorlinks=false,urlbordercolor=white]{hyperref}
  
\usepackage{hyperref}

\newcommand{\rsp}{\raisebox{0em}[2.7ex][1.3ex]{\rule{0em}{2ex} }}
\newtheorem{lem}{Lemma}
\newtheorem{prop}[lem]{Proposition}
\newtheorem{thm}[lem]{Theorem}

\newcommand{\Z}{{\mathbb Z}}

\newcommand{\Q}{{\mathbb Q}}

\newcommand{\Cl}{{\operatorname{Cl}}}

\newcommand{\gen}{{\operatorname{gen}}}

\newcommand{\eps}{\varepsilon}

\title{Octic Hilbert $2$-class fields of $\Q(\sqrt{2p}\,)$}

\author{Franz Lemmermeyer}
\email{franz.lemmermeyer@gmx.de}

\begin{document}

\maketitle

In this article we explain how to construct cyclic octic unramfied
extensions of the real quadratic number field $k = \Q(\sqrt{2p}\,)$, 
where $p \equiv 1 \bmod 8$ is
a prime number such that $h_2(k) \equiv 0 \bmod 8$. This can be done
by suitably modifying the construction over $\Q(\sqrt{-p}\,)$ given in
\cite{LemD}.

\section{Construction of Hilbert class fields of $\Q(\sqrt{2p}\,)$}\label{HC2}

We begin by collecting elementary results on the $2$-class group  of
$k = \Q(\sqrt{2p}\,)$ for primes $p \equiv 1 \bmod 4$.

\subsection*{Quadratic number fields $\Q(\sqrt{2p}\,)$}

Let $p \equiv 1 \bmod 4$ be a prime number. The number fields
$k = \Q(\sqrt{2p}\,)$ with $p \equiv 5 \bmod 8$ have class number
$h \equiv 2 \bmod 4$, and the Hilbert class field of $k$ coincides with
the genus class field $k_\gen = \Q(\sqrt{2}, \sqrt{p}\,)$.
If $p \equiv 1 \bmod 8$ then $d = 8 \cdot p$ is a $C_4$-factorization,
hence there exists a cyclic quartic extension $K/k$ unramified at all
finite primes. The construction of these fields is well known:
Write $p = e^2 - 2f^2$ with $e + f \sqrt{2} \equiv 1$ or
$\equiv 3 + 2\sqrt{2} \bmod 4$; then $K = k(\sqrt{e+f\sqrt{2}}\,)$.

\begin{table}[ht!]
  $$ \begin{array}{r|rr|cc}
    \rsp  p  & e & f & h & h^+  \\ \hline
     17 &  -5 &  2  & 2 &  4  \\
     73 &  -9 &  2  & 2 &  4  \\
     89 & -17 & 10  & 2 &  4  \\
     97 &  13 &  6  & 2 &  4  \\
    193 & -29 & 18  & 2 &  4     
  \end{array} \qquad \hskip 1cm
  \begin{array}{r|rr|cc}
    \rsp  p  & e & f & h & h^+  \\ \hline
    41 &  7 & 2  & 4 &  4  \\
    113 &  11 & 2  & 8 &  8  \\
    137 &  23 & 14  & 4 &  4  \\
    257 &  35 & 22  & 4 &  8  \\
    313 &  31 & 18  & 4 &  4  \\
  \end{array} $$
  \caption{Fields $\Q(\sqrt{2p}\,)$, where $p = e^2 - 2f^2$ with
    $f \equiv 2 \bmod 4$. Left: fields with $h \equiv 2 \bmod 4$;
    right: fields with $h$ divisible by $4$}
\end{table}

\subsection*{Class number divisibility}

Next we determine criteria for the divisibility of the class number of
$k = \Q(\sqrt{2p}\,)$ by powers of $2$ in terms of the representation
$p = e^2 - 2f^2$.

The $2$-class group of $k = \Q(\sqrt{2p}\,)$ is cyclic by genus theory.
If $p \equiv 1 \bmod 4$, then the genus class field of $k$ is
$k_\gen = \Q(\sqrt{2}, \sqrt{p}\,)$, hence the class number in the
usual sense is always even in this case.

\begin{prop}
  Let $p \equiv 1 \bmod 8$ be a prime number. Then $p = e^2 - 2f^2$
  for integers $e$ and $f > 0$ with $e \equiv 3 \bmod 4$ and
  $f \equiv 2 \bmod 4$.

  $$ \begin{array}{lc|l}
    \rsp \text{class number} &  N\eps_{2p} & \text{conditions} \\ \hline
    \rsp  h \equiv 2 \bmod 4 & +1    & e < 0 \\
    \rsp  h \equiv 4 \bmod 8 & -1    & e > 0, \ e \equiv 7 \bmod 8 \\
    \rsp  h \equiv 4 \bmod 8 & +1    & e > 0, \ e \equiv 3 \bmod 8 \\
    \rsp  h \equiv 0 \bmod 8 & \pm 1 & e > 0, \ e \equiv 3 \bmod 8
  \end{array} $$
\end{prop}

\begin{proof}
  The extension $k(\sqrt{\alpha}\,)$ for $\alpha = e + f\sqrt{2}$
  is a cyclic quartic extension of $k = \Q(\sqrt{2p}\,)$ containing
  $k_\gen = \Q(\sqrt{2},\sqrt{2p}\,)$. Since
  $N(\alpha) = \alpha\alpha' = e^2 - 2f^2 = p$, it can only ramify above
  the primes dividing $2 p \infty$. Since
  $k(\sqrt{\alpha}\,) = k(\sqrt{\alpha'}\,)$ and since $\alpha$ and $\alpha'$
  are coprime,  it can only ramify at $2$ and $\infty$.

  Since $\alpha \equiv 3+2\sqrt{2} = (1+\sqrt{2}\,)^2 \bmod 4$, the
  extension $k(\sqrt{\alpha}\,)$ is a cyclic quartic extension
  unramified except possibly at the infinite primes.

  Thus if $e < 0$, then the $2$-class number of $k$ is $h = 2$ and the
  $2$-class number of $k$ in the strict sense is $h^+ = 4$; this proves the
  first line.

  If $e > 0$ and $e \equiv 7 \bmod 8$, then $(\frac p2)_4 = (\frac2p)_4 = -1$
  by Lemma~\ref{Lem1} below, hence $h = h^+ = 4$ by Prop.~\ref{Pr01}.

  If $e > 0$ and $e \equiv 3 \bmod 8$, then $(\frac p2)_4 = (\frac2p)_4 = +1$,
  hence $h^+$ is divisible by $8$. 
\end{proof}

\begin{prop}\label{Pr01}
  Let $p$ be an odd prime number. The class number $h^+$ in the strict
  sense of $\Q(\sqrt{2p}\,)$ is divisible by $8$ if and only if
  $p \equiv 1 \bmod 16$ and $(\frac 2p)_4 = +1$.
  We have $h \equiv 4 \bmod 8$ and $N\eps_{2p} = -1$  if and only if
  $p \equiv 9 \bmod 16$ and $(\frac 2p)_4 = -1$.
\end{prop}

\begin{proof}
  This is a special case of Scholz's reciprocity law.  
\end{proof}

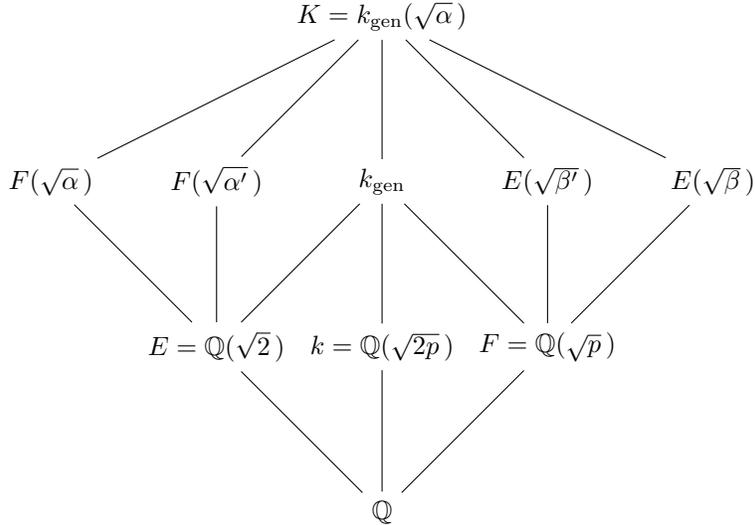
\begin{figure}[ht!]
  \begin{tikzpicture}[node distance=2.2cm,scale=0.7]
  \node (Q) at (0,0) {$\Q$};
  \node (k) [above of=Q] {$k=\Q(\sqrt{2p}\,)$};
  \node (F) [left  of=k] {$E=\Q(\sqrt{2}\,)$};
  \node (E) [right of=k] {$F=\Q(\sqrt{p}\,)$};
  \node (K) [above of=k] {$k_\gen$};
  \node (F41) [left  of=K]   {$F(\sqrt{\alpha'}\,)$};
  \node (F42) [left  of=F41] {$F(\sqrt{\alpha}\,)$};
  \node (E41) [right  of=K]   {$E(\sqrt{\beta'}\,)$};
  \node (E42) [right of=E41] {$E(\sqrt{\beta}\,)$};
  \node (L) [above of=K] {$K=k_\gen(\sqrt{\alpha}\,)$};
  \draw (Q) -- (k);
  \draw (Q) -- (E);
  \draw (Q) -- (F);
  \draw (k) -- (K);
  \draw (F) -- (K);
  \draw (F) -- (F41);
  \draw (F) -- (F42);
  \draw (E) -- (K);
  \draw (E) -- (E41);
  \draw (E) -- (E42);
  \draw (K) -- (L);
  \draw (F41) -- (L);
  \draw (F42) -- (L);
  \draw (E41) -- (L);
  \draw (E42) -- (L);
  \end{tikzpicture}
  \caption{Subfields of $L = k(\sqrt{\alpha}\,)$}
\end{figure}

\begin{lem}\label{Lem1}
  Let $p = e^2 - 2f^2 \equiv 1 \bmod 8$ be a prime and assume that
  $f \equiv 2 \bmod 4$.  Then
  $$ \Big(\frac p2 \Big)_4 = \Big(\frac p2 \Big)_4 = \Big(\frac{-2}e \Big). $$
\end{lem}

Write $p = e^2 - 2f^2$; if $f$ is divisible by $4$, then
replace $e + f \sqrt{2}$ by
$$ e' + f'\sqrt{2} = (e + f \sqrt{2}\,)(3 + 2\sqrt{2}\,)
   = 3e+4f + (2e+3f) \sqrt{2}; $$
then $f' \equiv 2 \bmod 4$. Thus we may assume that $f \equiv 2 \bmod 4$.
This implies $p \equiv e^2 - 8 \bmod 16$. Since $p \equiv 1 \bmod 16$ when
$h$ is divisible by $8$ we must have $e \equiv \pm 3 \bmod 8$.
On the other hand reducing $p = e^2 - 2f^2$ modulo $e$ we find
$(\frac pe) = (\frac{-2}e)$; since $(\frac ep) = (\frac{2}p)_4 (\frac fp)$
and, writing $f = 2g$,
$(\frac fp) = (\frac{2g}p) = (\frac{g}p) = (\frac pg) = +1$.
Thus $(\frac{-2}e) = +1$, and hence $e \equiv 3 \bmod 8$.

Let us summarize our main result concerning the cyclic quartic unramified
extension of $k = \Q(\sqrt{2p}\,)$:

\begin{prop}
  Let $p = e^2 - 2f^2 \equiv 1 \bmod 8$ be a prime number, and choose
  $e$ and $f$ in such a way that $e \equiv 3 \bmod 4$ and
  $f \equiv 2 \bmod 4$. Then $k = \Q(\sqrt{2p}\,)$ admits a cyclic
  quartic extension $K = k(\sqrt{\alpha}\,)$ with $\alpha = e + f\sqrt{2}$
  unramified at all finite primes. In addition we have:
  \begin{itemize}
  \item $K$ is totally real if and only if one of the following equivalent
    conditions holds:
    \begin{enumerate}
    \item The class number $h_k$ is divisible by $4$;
    \item $e > 0$;
    \end{enumerate}
  \item $K$ is totally complex if and only if one of the following equivalent
    conditions holds:
    \begin{enumerate}
    \item $h_k \equiv 2 \bmod 4$ and $N\eps_{2p} = +1$;
    \item $e < 0$.
    \end{enumerate}
  \end{itemize}  
\end{prop}

If $h_k \equiv 2 \bmod 4$ and $N\eps_{2p} = -1$, then the Hilbert
$2$-class field of $k$ is the genus field $k_\gen = \Q(\sqrt{2}, \sqrt{p}\,)$;
in this case there is no cyclic quartic unramified extension of $k$.

Our next goal is constructing an unramified cyclic octic extension of $k$ by
finding a quadratic unramified extension of $K$.

\section{Cyclic octic extensions}

Cyclic octic extensions of $\Q(\sqrt{2p}\,)$ unramified at all finite
primes exist if and only if the class number $h^+$ in the strict sense is
divisible by $8$; this holds if either $h \equiv 0 \bmod 8$, or if
$h \equiv 4 \bmod 8$ and $N\eps_{2p} = +1$.

$L = \Q(\sqrt{2},\sqrt{p}\,)$ is a quadratic extension of the
reasonable\footnote{A number field is called reasonable if it is totally
  real and has odd class number in the strict sense; see \cite{LTh,LemD}.}
field $F = \Q(\sqrt{2}\,)$. Let $K = k_\gen(\sqrt{\alpha}\,)$ be the unramified
cyclic quartic extension of $k$ unramified at all finite primes. 
By the results of \cite{LTh}, such an extension exists if and only if
$p = \alpha \cdot \alpha'$ is a $C_4$-factorization, i.e., if and only
if $[\alpha/\alpha'] = +1$.
In our case, this condition is satisfied since 
$[\alpha/\alpha'] = [\alpha + \alpha'/\alpha'] = [2e/\alpha'] =
 (2e/p) = (e/p) = (p/e) = (-2/e) = +1$.

In this case, the diophantine equation
\begin{equation}\label{EqC4} A^2 - \alpha B^2 = \alpha'C^2 \end{equation}
has nontrivial solutions in $\Z[\sqrt{2}\,]$, and then $L(\sqrt{\mu}\,)$
with $\mu = A+B\sqrt{\alpha}$ is a cyclic octic extension of
$\Q(\sqrt{2p}\,)$ unramified outside $2 \infty$; the proof of these
claims is well known (see e.g. \cite{LemD}). We know from \cite{LTh}
(see also \cite[Cor. 5.3]{LemDir}) that  
there exists a unit $\eps \in F = \Q(\sqrt{2}\,)$ such that
$L\big(\sqrt{(A+B\sqrt{\alpha}\,)\eps}\big)$ is a cyclic octic extension of
$\Q(\sqrt{2p}\,)$ unramified outside $\infty$. 
We will now explain how to solve the equation (\ref{EqC4}) and how to
find this unit $\eps$.

For solving the equation (\ref{EqC4}) for $\alpha = e + f\sqrt{2}$ we can take
$$ A = u\sqrt{2}, \quad B = r + s \sqrt{2}, \quad C = r - s\sqrt{2} $$
for integers $r, s, u$ with $u$ odd. Then
$$ u^2 = er^2 + 4rsf + 2es^2, $$
hence
$$ eu^2 = (er)^2 + 4fs(er) + 2e^2s^2 = (er+2fs)^2 + 2(e^2 - 2f^2)s^2. $$
Since $e^2 - 2f^2 = p$ we now have the diophantine equation
\begin{equation}\label{EDZ1} eu^2 = t^2 + 2ps^2 \end{equation}
with $t = er + 2fs$.

\begin{lem}
  If $e > 0$, $e \equiv 3 \bmod 8$ and $f \equiv 2 \bmod 4$, then the
  diophantine equation (\ref{EDZ1}) is solvable in odd integers
  $s$, $t$ and $u$.
\end{lem}

\begin{proof}
  If the equation is solvable with $t$ even then $u$ and $s$ are
  also even; dividing through by $2$ sufficiently often we obtain a solution
  in odd integers. Thus we only have to show that the equation has a
  nontrivial rational solution (multiplying through by the common
  denominator yields an integral solution).

  For showing that $eu^2 = t^2 + 2ps^2$ has a nonzero rational solution
  we have to show solvability modulo the reals (which is equivalent to
  $e > 0$), modulo $8$ (which follows from $e \equiv 1 + 2p \equiv 3 \bmod 8$),
  modulo $p$ and modulo $e$.

  Solvability modulo $p$ is equivalent to $(\frac ep) = +1$.
  But $(\frac ep) = (\frac pe) = (\frac{-2}e) = +1$ since $p  = e^2 - 2f^2$.

  Since $t^2 \equiv -2ps^2 \bmod e$ we have to show that
  $(\frac{-2p}{e}) = +1$.
  But $p = e^2 - 2f^2$ implies $(\frac{p}e) = (\frac{-2}e) = +1$.   
\end{proof}

\begin{table}[ht!]
$$ \begin{array}{r|cc|cc|rrrr}
    \rsp  p  & h & N\eps_{2p} & e & f  & u & t & s & r \\ \hline
    \rsp  113 & 8 & - 1 & 11 &  2 &   5 &   7 & -1 & 1 \\
    \rsp  257 & 4 & \ 1 & 35 & 22 &  11 &  61 & -1 & 3 \\
    \rsp  337 & 4 & \ 1 & 27 & 14 &   5 &   1 &  1 & -1 \\
    \rsp  353 & 4 & \ 1 & 19 &  2 &   7 &  15 & -1 &  1 \\ 
    \rsp 1201 & 8 & - 1 & 43 & 18 &  37 & 193 & -3 &  7 \\
    \rsp 1217 & 8 & \ 1 & 35 &  2 &  19 & 101 & -1 &  3 \\
    \rsp 1601 & 8 & - 1 & 67 & 38 &   7 &   9 &  1 & -1 \\
    \rsp 1777 & 8 & - 1 & 43 &  6 &  47 & 251 &  3 &  5 \\
    \rsp 2113 & 8 &  -1 & 99 & 62 &   7 &  25 &  1 & -1
  \end{array} $$  
  \caption{Fields $\Q(\sqrt{2p}\,)$ with narrow class number divisible by $8$}
\end{table}

Note that
$$ eu^2 = t^2 + 2ps^2 = t^2 + 2(e^2 - 2f^2)s^2 = t^2 + 2e^2 - 4f^2s^2 $$
implies that $t^2 \equiv (2fs)^2 \bmod e$. 
If we choose the sign of $s$ in such a way that $t \equiv 2fs \bmod e$, 
then $r = \frac{t-2fs}{e}$ is an integer. This is always possible if $e$
is prime; if $e$ is composite, such a choice is not always possible.

Consider e.g. $p = 2593$ with $e = 51$ and $f = 2$; here $u = 19$,
$t = 115$ and $s = 1$. Although $t^2 \equiv (2fs)^2 \equiv 16 \bmod 51$,
we have $\gcd(t-2fs,e) = 3$ and $\gcd(t+2fs,e) = 17$. The same
problem occurs for the solution $(u,t,s)=(353, 47, 35)$; with
$(u,t,s)= (75, 181, 7)$, we do get an integral value for $r$.

If we work with the first solution and clear denominators, we get the
required extension:
$$ \mu = \Big[57\sqrt{2} + (7 - 3\sqrt{2}\,) \sqrt{51 + 2\sqrt{2}}\Big]
         (1+\sqrt{2}\,). $$

\bigskip

Recall that there is a unique unit $\eps \in \{\pm 1, \pm 1 \pm \sqrt{2}\}$
such that $\mu \eps$ generates a cyclic octic unramified extension.
For determining the correct unit it suffices to determine the
signature of
$$ \alpha = A + B \sqrt{\mu} = u\sqrt{2} + (r+s\sqrt{2}\,) \sqrt{\mu}. $$
If we choose $u > 0$, then the relative norm of $\alpha$ is positive:
$$ \alpha \alpha' =  A^2 - \mu B^2 = \mu' C^2. $$
With $\beta = -u\sqrt{2} + (r-s\sqrt{2}\,) \sqrt{\mu'}$
we have $\beta, \beta' < 0$. 
Thus the signature of $\alpha$ is $(++--)$, where the first two embeddings
correspond to the positive square root of $2$. Multiplication
by $\pm 1 - \sqrt{2}$ provides us with a totally negative element
$\mu = \alpha(1 - \sqrt{2}\,)$, and if we multiply by $\pm 1 + \sqrt{2}$
we get a totally positive element. But since
$-1 + \sqrt{2} = \frac{1+\sqrt{2}}{(1+\sqrt{2}\,)^2}$, both elements
give rise to the same extension.

Uising the congruence
\begin{align*}
  \Big(\frac{1 + \sqrt{e+f\sqrt{2}}}{\sqrt{2}}\Big)^2
  & \equiv \begin{cases}
    2 + \sqrt{2} + \sqrt{e+f\sqrt{2}} \bmod 4
                 & \text{ if } f \equiv 2 \bmod 8, \\
    2 - \sqrt{2} + \sqrt{e+f\sqrt{2}} \bmod 4
                 & \text{ if } f \equiv 6 \bmod 8,
  \end{cases}
\end{align*}
it can be verified that the elements whose square roots generate
the unramified octic extensions actually are congruent to a square
modulo $4$.

\begin{table}[ht!]
  $$ \begin{array}{r|c|c|c}    
    \rsp  p  & \mu & \Cl(k_\gen) & \Cl(k^1) \\ \hline
    \rsp 113 & \big[\,5 \sqrt{2} + (1 - \sqrt{2}\,)\sqrt{11 + 2\sqrt{2}}\,\big]
    (1 + \sqrt{2}\,)  & [4] & [1] \\
    \rsp 1201 & \big[37\sqrt{2} + (7 -3\sqrt{2}\,) \sqrt{43 + 18\sqrt{2}}\,\big]
    (1 + \sqrt{2}\,) & [4] & [1] \\
    \rsp 1217 & \big[19\sqrt{2} + (3 -\sqrt{2}\,) \sqrt{35+2\sqrt{2}}\,\big] 
    (1 + \sqrt{2}\,) & [4] & [1] \\  
    \rsp 1601 & \big[7\sqrt{2} + (1 - \sqrt{2}\,) \sqrt{67+38\sqrt{2}}\,\big]
    (1 + \sqrt{2}\,) & [28] & [7]   \\
    \rsp 1777 & \big[47\sqrt{2} + (5+3\sqrt{2}\,)\sqrt{43+6\sqrt{2}}\,\big]
    (1 + \sqrt{2}\,) & [4] & [1] \\
    \rsp 2113 & \big[ 7\sqrt{2} + (1-\sqrt{2}\,) \sqrt{99+62\sqrt{2}}\big]
    (1 + \sqrt{2}\,) & [4] & [3,3] \\
  \end{array} $$
  \caption{Totally real unramified cyclic octic extensions}\label{Tab3}
\end{table}

\begin{table}[ht!]
  $$ \begin{array}{r|c|c|c}    
    \rsp  p  & \mu & \Cl(k_\gen) & \Cl(k^{+\,1}) \\ \hline
    \rsp 257 & \big[11\sqrt{2} + (3 - \sqrt{2}\,)\sqrt{35 + 22\sqrt{2}}\,\big]
    (1 - \sqrt{2}\,) & [6] & [3] \\
    \rsp 337 & \big[ 5\sqrt{2} + (1 - \sqrt{2}\,)\sqrt{27 + 14\sqrt{2}}\,\big]
    (1 - \sqrt{2}\,) & [2] & [17,17] \\
   \rsp 353 & \big[ 7\sqrt{2} + (1 - \sqrt{2}\,)\sqrt{19 + 2\sqrt{2}}\, \big]
    (1 - \sqrt{2}\,) & [2] & [17,17] \\
   \rsp 577 &  \big(13\sqrt{2} + (3 - \sqrt{2}\,)\sqrt{35 + 18\sqrt{2}}\,\big]
    (1 - \sqrt{2}\,) & [14] & [31, 31, 7] \\
   \rsp 593 & \big[ 5\sqrt{2} + (1 - \sqrt{2}\,) \sqrt{59 + 38\sqrt{2}}\,\big]
    (1 - \sqrt{2}\,) & [2] & [31, 31] \\
   \rsp 881 & \big[23\sqrt{2} + (5 - 3\sqrt{2}\,) \sqrt{43+22\sqrt{2}}\,\big]
   (1 - \sqrt{2}\,) & [2] &  [41, 41] \\
   \rsp 1153 & \big[9\sqrt{2} + (1  -\sqrt{2}\,) \sqrt{35 + 6\sqrt{2}}\,\big]
   (1 - \sqrt{2}\,) & [2] &  [7, 7] \\
    \rsp 1249 & \big[7\sqrt{2} + (1 - \sqrt{2}\,) \sqrt{51+26\sqrt{2}}\,\big]
    (1 - \sqrt{2}\,) & [2] & [71, 71] \\    
    \rsp 1553 & \big[35\sqrt{2} + (7 -3\sqrt{2}\,) \sqrt{91+58\sqrt{2}}\,\big]
    (1 - \sqrt{2}\,) & [2] & [7, 7]
 \end{array} $$
  \caption{Totally complex unramified octic cyclic extensions}\label{Tab4}
\end{table}

\begin{thm}
  Let $p \equiv 1 \bmod 8$ be a prime with $p = e^2 - 2f^2$, where
  $e \equiv 3 \bmod 8$ and $f \equiv 2 \bmod 4$. Then the equation
  \begin{equation}\label{E2p1}
    eu^2 = t^2 + 2ps^2
  \end{equation}
  is solvable in nonzero integers. There is a solution for which we can
  Choose the sign of $s$ in such a way 
  that $t \equiv 2fs \bmod e$; let $g$ be the deominator of
  $r = \frac{t-2fs}{e}$ when written in lowest form. Then for
  $$ \mu = g\big(u\sqrt{2} + (r+s\sqrt{2}\,) \sqrt{\mu}\,\big)\eps, $$
  where $\eps = 1 \pm \sqrt{2}$ and
  $K = \Q(\sqrt{2},\sqrt{p}\,)$, the extensions $K(\sqrt{\mu}\,)$
  are cyclic quartic extension unramified outside $2 \infty$.
  If we choose
  $$ \eps  = \begin{cases}
    1 + \sqrt{2} & \text{ if } h \equiv 0 \bmod 8, \\
    1 - \sqrt{2} & \text{ if } h \equiv 4 \bmod 8, \end{cases} $$
  then $K(\sqrt{\mu \eps}\,)$ is the cyclic octic extension of
  $\Q(\sqrt{2p}\,)$ unramified at all finite primes.
\end{thm}

\end{document}